\documentclass[a4paper,10pt]{amsart}
\usepackage[arrow,matrix]{xy}
\usepackage{color}
\usepackage{amsmath,amssymb,amscd,bbm,amsthm,mathrsfs,dsfont,enumerate}
\theoremstyle{plain} \textwidth=36pc \textheight=51pc

\newtheorem{theorem}{Theorem}[section]
\newtheorem{lemma}[theorem]{Lemma}

\newtheorem{proposition}[theorem]{Proposition}
\newtheorem{corollary}[theorem]{Corollary}
\theoremstyle{definition}

\numberwithin{equation}{section}

\newcommand{\K}{\bold{k}}

\newcommand{\Char}{\mathop{\mathrm{Char}}}
\newcommand{\rank}{\mathop{\mathrm{rank}}}

\begin{document}

\title[Universal enveloping algebras of Poisson Ore extensions]{Universal enveloping algebras of Poisson Ore extensions}

\author{Jiafeng L\"u, Xingting Wang and Guangbin Zhuang}

\address{(L\"u) Department of Mathematics, Zhejiang Normal University, Jinhua, Zhejiang 321004, P.R. China}
\email{jiafenglv@gmail.com}

\address{(Wang) Department of Mathematics,
University of Washington, Seattle, Washington 98195, USA}
\email{xingting@uw.edu}

\address{(Zhuang) Department of Mathematics,
University of Southern California, Los Angeles 90089-2532, USA}
\email{gzhuang@usc.edu}



\thanks{The first author is supported by National Natural Science Foundation of China (No. 11001245, 11271335 and 11101288). The second author is partially supported by U.~S.~National Science Foundation [DMS0855743].}

\begin{abstract}
We prove that the universal enveloping algebra of a Poisson-Ore extension is a length two iterated Ore extension of the original universal enveloping algebra. As consequences, we observe certain ring-theoretic invariants of the universal enveloping algebras that are preserved under iterated Poisson-Ore extensions. We apply our results to iterated quadratic Poisson algebras arising from semiclassical limits of quantized coordinate rings and a family of graded Poisson algebras of Poisson structures of rank at most two.
\end{abstract}
\maketitle

\section*{Introduction}
As an analogue of the classical enveloping algebras of Lie algebras, the notion of Poisson universal enveloping algebra was first introduced in \cite{Oh1} to illustrate the equivalence of the following two categories:
\begin{align*}
\mathbf{PMod}(R)\equiv \mathbf{Mod}(R^e),
\end{align*}
which translates the representations of a Poisson algebra $R$ into the representations of a noncommutative algebra $R^e$, called the Poisson universal enveloping algebra of $R$. Natural questions arise about the structures between $R$ and $R^e$, to which our main result is stated regarding Poisson-Ore extensions in the sense of \cite{Oh3}.
\begin{theorem}\label{Thm:1}
Let $R$ be a Poisson algebra, and $R^e$ be its universal enveloping algebra. 
\begin{enumerate}
\item For any Poisson-Ore extension $A$ of $R$, the universal enveloping algebra $A^e$ is a right double Ore extension of $R^e$. Moreover, it is a length two iterated Ore extension.
\item For any iterated Poisson-Ore extension $A$ of $R$, the universal enveloping algebra $A^e$ is an iterated Ore extension of $R^e$ of double length. 
\end{enumerate}
\end{theorem}
Our proof strategy is first to show that $A^e$ is a right double Ore extension of $R^e$ (Subsection \ref{DOE}), and then apply \cite[Theorem 2.4]{CLM:11} to show that it is indeed a length two iterated Ore extension. As consequences, we have the following corollary by using the properties of Ore extensions.
\begin{corollary}\label{Cor:1}
Let $R$ be a Poisson algebra, and $A$ be an iterated Poisson-Ore extension of $R$. Then the universal enveloping algebra $A^e$ inherits the following properties from $R^e$:
\begin{enumerate}
\item being a domain;
\item being Noetherian;
\item having finite global dimension;
\item having finite Krull dimension;
\item being twisted Calabi-Yau.
\end{enumerate}
In particular, let $R$ be a connected graded Poisson algebra, and $A$ be a graded iterated Poisson-Ore extension of $R$. Then $R^e$ and $A^e$ are connected graded algebras, and $A^e$ aslo inherits the following properties from $R^e$:
\begin{enumerate}
\item[(6)] being Artin-Schelter regular;
\item[(7)] being Koszul provided that $A$ is quadratic.
\end{enumerate}
\end{corollary}

For applications, we consider iterated quadratic Poisson algebras arising from semiclassical limits of quantized coordinate rings. In fact, these quantum algebras are all iterated Ore extensions of the polynomial algebra with one variable. Hence they give rise to iterated Poisson-Ore extensions of the same algebra with trivial Poisson bracket through the semiclassical limit process, see reference \cite[Proposition 4.1]{LL:13}. Also, we consider one family of graded Poisson algebras, whose Poisson bracket is determined by a skew-symmetric matrix. Our classification shows that their Poisson structures have rank at most two in the sense of \cite[\S 3]{BG}, and the isomorphism classes consist of a discrete class and a parametric family, where the latter one is a Poisson-Ore extension of a free Poisson algebra.

\subsection*{Acknowledgments} The authors first want to give their sincere gratitudes to James Zhang for introducing them this project. They also want to thank Yanhong Bao, Jiwei He, Xuefeng Mao, Cris Negron and James Zhang for many valuable discussions and suggestions on this paper.

\section{Preliminaries}
Throughout, we work over a base field $\K$. We use $R$ to denote a commutative algebra, and it is said to be connected graded if $R=\K\oplus R_1\oplus R_2\oplus \cdots$ such that $R_iR_j\subseteq R_{i+j}$ for all $i,j\ge 0$. 
\subsection{Poisson algebras}
A Poisson algebra is a commutative algebra $R$ with a Lie bracket $\{-,-\}$ such that
\[
\{ab,c\}=a\{b,c\}+\{a,c\}b
\]
for all $a,b,c\in R$, called Leibniz rule. Also a graded Poisson algebra is defined to be a graded commutative algebra with a degree $0$ graded Lie bracket satisfying the Leibniz rule.

\subsection{Universal enveloping algebras of Poisson algebras}\label{UniversalProperty}
The universal enveloping algebra of a Poisson algebra $R$ is determined by the following universal property: let $(A, m, h)$ be a triple, which has property $\mathbf{P}$ described as
\begin{enumerate}
\item[(P1)] $A$ is an algebra and $m: R\to A$ is an algebra map;
\item[(P2)] $h: (R,\{-,-\})\to A_L$ is a Lie algebra map;
\item[(P3)] $m_{\{r,s\}}=h_rm_s-m_sh_r$, and
\item[(P4)] $h_{rs}=m_rh_s+m_sh_r$, for all $r,s\in R$,
\end{enumerate}
then $(A, m, h)$ is called the universal enveloping algebra of $R$ if for any other triple $(B,f,g)$ satisfying property $\mathbf{P}$, there exists a unique algebra map $\phi: A\to B$ such that $f=\phi m$ and $g=\phi h$, see reference \cite{Oh1}. Sometimes, we say that the following diagram commutes with respect to property $\mathbf{P}$: 
\[
\xymatrix{
R\ar[rr]^-{m}_-{h}\ar[dr]^-{f}_-{g} && A\ar@{-->}[dl]^-{\exists ! \phi}\\
& B&
}.
\]
The universal enveloping algebra of $R$ is usually denoted by $R^e$, which can be constructed explicitly. We follow \cite[\S 2]{U}. Let $V=R\oplus R$ with two inclusions of $R$ denoted by $m$ and $h$. Denote by $R^e$ the tensor algebra $T(V)$ modulo the following relations:
\begin{align*}
m_{rs}=m_rm_s,\tag{R1}\\
h_{\{r,s\}}=h_rh_s-h_sh_r,\\
h_{rs}=m_rh_s+m_sh_r,\\
m_{\{r,s\}}=h_rm_s-m_sh_r=[h_r, m_s],\\
m_1=1,
\end{align*}
for all $r,s\in R$. It is clear that $m, h$ induce two linear maps from $R$ to $R^e$, where we keep the same notations. Then the universal enveloping algebra of $R$ is given by the triple $(R^e,m,h)$. Sometimes, we only call $R^e$ the universal enveloping algebra of $R$ without specifying the two linear maps $m,h$. Note that $R^e$ is generated by $m_R$ and $h_R$ as an algebra. Furthermore, if $R$ is a graded Poisson algebra, then so is $R^e$ by setting $\deg m_r=\deg h_r=\deg r$ for any homogenous element $r\in R$. 

\subsection{Poisson-Ore extensions}\label{subsection:POE}
Let $R$ be a Poisson algebra. A linear map $\alpha: R\to R$ is said to be a Poisson derivation if it satisfies
\begin{enumerate}
\item $\alpha(rs)=\alpha(r)s+r\alpha(s)$;
\item $\alpha(\{r,s\})=\{\alpha(r),s\}+\{r,\alpha(s)\}$,
\end{enumerate}
for all $r,s\in R$. Let $\alpha$ be a Poisson derivation of $R$, then a linear map $\delta: R\to R$ is called a Poisson $\alpha$-derivation if it satisfies
\begin{enumerate}
\item $\delta(rs)=\delta(r)s+r\delta(s)$;
\item $\delta(\{r,s\})=\{\delta(r),s\}+\{r,\delta(s)\}+\alpha(r)\delta(s)-\delta(r)\alpha(s)$,
\end{enumerate}
for all $r,s\in R$. 
\begin{theorem}\cite[Theorem 1.1]{Oh3}\label{POE}
Let $\alpha$ and $\delta$ be two linear maps of a Poisson algebra $R$. Then the polynomial algebra $A=R[x]$ is a Poisson algebra with Poisson bracket extending the Poisson bracket of $R$ such that
\[
\{x,r\}=\alpha(r)x+\delta(r),
\]
for all $r\in R$ if and only if $\alpha$ is a Poisson derivation of $R$ and $\delta$ is a Poisson $\alpha$-derivation of $R$.
\end{theorem} 
The algebra $A$ endowed with the Poisson bracket from Theorem \ref{POE} is denoted by $A=R[x;\alpha,\delta]_P$ and called Poisson-Ore extension of $R$.

\subsection{Double Ore extensions}\label{DOE}
We follow \cite[Definition 1.3, Lemma 1.10, Proposition 1.11]{ZZ:08}. The right double Ore extension of an associative algebra $A$  is adding two generators $y_1,y_2$ to $A$, subject to
\begin{enumerate}
\item[(D1)] 
$\begin{pmatrix}
y_1\\
y_2
\end{pmatrix}
a=\sigma(a)\begin{pmatrix}
y_1\\
y_2
\end{pmatrix}+\eta(a)$, for all $a\in A$, where $\sigma=\begin{pmatrix}
\sigma_{11}& \sigma_{12}\\
\sigma_{21}& \sigma_{22}
\end{pmatrix}: A\to M_2(A)$ is an algebra map and $\eta=\begin{pmatrix} \eta_1\\ \eta_2\end{pmatrix}: A\to A^{\oplus 2}$ is a $\sigma$-derivation satisfying $\eta(ab)=\sigma(a)\eta(b)+\eta(a)b$, for all $a,b\in A$;
\item[(D2)] $y_2y_1=p_{12}y_1y_2+p_{11}y_1^2+\tau_1y_1+\tau_2y_2+\tau_0$, where $P:=\{p_{12}, p_{11}\}$ is a set of elements of $\K$ and $\tau:=\{\tau_0,\tau_1,\tau_2\}$ is a set of elements of $A$;
\item[(D3)] $6$ compatible conditions:
\begin{align*}
\sigma_{21}\sigma_{11}+p_{11}\sigma_{22}\sigma_{11}&=p_{11}\sigma_{11}^2+p_{11}^2\sigma_{12}\sigma_{11}+p_{12}\sigma_{11}\sigma_{21}+p_{11}p_{12}\sigma_{12}\sigma_{21},\\
\sigma_{21}\sigma_{12}+p_{12}\sigma_{22}\sigma_{11}&=p_{11}\sigma_{11}\sigma_{12}+p_{11}p_{12}\sigma_{12}\sigma_{11}+p_{12}\sigma_{11}\sigma_{22}+p_{12}^2\sigma_{12}\sigma_{21},\\
\sigma_{22}\sigma_{12}&=p_{11}\sigma_{12}^2+p_{12}\sigma_{12}\sigma_{22},\\
\sigma_{20}\sigma_{11}+\sigma_{21}\sigma_{10}+\sigma_{22}\sigma_{11}\tau_1&=p_{11}(\sigma_{10}\sigma_{11}+\sigma_{11}\sigma_{10}+\tau_1\sigma_{12}\sigma_{11})+p_{12}(\sigma_{10}\sigma_{21}+\sigma_{11}\sigma_{20}+\tau_1\sigma_{12}\sigma_{21})\\ &\quad+\tau_1\sigma_{11}+\tau_2\sigma_{21},\\
\sigma_{20}\sigma_{12}+\sigma_{22}\sigma_{10}+\sigma_{22}\sigma_{11}\tau_2&=p_{11}(\sigma_{10}\sigma_{12}+\sigma_{12}\sigma_{10}+\tau_2\sigma_{12}\sigma_{11})+p_{12}(\sigma_{10}\sigma_{22}+\sigma_{12}\sigma_{20}+\tau_2\sigma_{12}\sigma_{21})\\& \quad+\tau_1\sigma_{12}+\tau_2\sigma_{22},\\
\sigma_{20}\sigma_{10}+\sigma_{22}\sigma_{11}\tau_0&=p_{11}(\sigma_{10}^2+\tau_0\sigma_{12}\sigma_{11})+p_{12}(\sigma_{10}\sigma_{20}+\tau_0\sigma_{12}\sigma_{21})+\tau_1\sigma_{10}+\tau_2\sigma_{20}+\tau_0Id,
\end{align*}  

where $\sigma_{i0}=\eta_i$ for $i=1,2$.
\end{enumerate}

We denote by $A_P[y_1,y_2;\sigma,\eta,\tau]$ the right double Ore extension of $A$ associated to the DE-data $\{P,\sigma,\eta,\tau\}$. By symmetry, we have the notion of left double Ore extension, and we say that the extension is a double Ore extension if it can be obtained by adding two generators via both right and left double Ore extensions.

\section{Universal enveloping algebras of Poisson-Ore extensions}
In this section, we will show that the universal enveloping algebra of a Poisson-Ore extension is a length two iterated Ore extension of the original universal enveloping algebra. Let $R$ be a Poisson algebra, and we consider its Poisson-Ore extension $R[x;\alpha,\delta]_P$ for some Poisson derivation $\alpha$ of $R$ and some Poisson $\alpha$-derivation $\delta$ of $R$. The Poisson bracket in $R[x;\alpha,\delta]_P$ can be explicitly given by
\begin{align}\label{e1}
\left\{rx^p,sx^q\right\}=\left(\{r,s\}+pr\alpha(s)-q\alpha(r)s\right)x^{p+q}+\left(pr\delta(s)-q\delta(r)s\right)x^{p+q-1},
\end{align}
for all $r,s\in R$ and $p,q\ge 0$. We denote by $(R^e,m,h)$ the universal enveloping algebra of $R$, which has property $\mathbf{P}$. 

First of all, we state the DE-data $\{P,\sigma,\eta,\tau\}$ for the right double Ore extension of $R^e$ in Theorem \ref{Thm:1}.
\begin{itemize}
\item $\sigma: R^e\to M_2(R^e)$ is given by 
\begin{align}\label{e3}
\sigma(m_r)=
\begin{pmatrix}
m_r  & 0\\
m_{\alpha(r)} & m_r
\end{pmatrix},\quad
\sigma(h_r)=
\begin{pmatrix}
h_r+m_{\alpha(r)}  & 0\\
h_{\alpha(r)}+m_{\alpha^2(r)} & h_r+m_{\alpha(r)}
\end{pmatrix},\ \mbox{for all}\ r\in R.
\end{align}
\item The $\sigma$-derivation $\eta: R^e\to (R^e)^{\oplus 2}$ is defined by
\begin{align*}
\eta(m_r)=\begin{pmatrix} 0\\ m_{\delta(r)}\end{pmatrix},\quad \eta(h_r)=\begin{pmatrix} m_{\delta(r)} \\ h_{\delta(r)}+m_{\delta\alpha(r)}\end{pmatrix},\ \mbox{for all}\ r\in R.
\end{align*}
\item  $p_{11}=0, p_{12}=1$ in $P$ and $\tau=0$, i.e., $y_2y_1=y_1y_2$.
\end{itemize}

Secondly, we will show that the DE-data given above is well-defined. For (D1), we define two linear maps $f,g: R\to M_2(R^e)$ by
\begin{align}\label{e2}
f(r)=
\begin{pmatrix}
m_r  & 0\\
m_{\alpha(r)} & m_r
\end{pmatrix},\quad
g(r)=
\begin{pmatrix}
h_r+m_{\alpha(r)}  & 0\\
h_{\alpha(r)}+m_{\alpha^2(r)} & h_r+m_{\alpha(r)}
\end{pmatrix},
\end{align}
for all $r\in R$. According to Theorem \ref{POE}, it is direct to check that the triple $(M_2(R^e),f,g)$ has property $\mathbf{P}$.
Hence the universal property of $R^e$ guarantees the existence of a unique algebra map $\sigma: R^e\to M_2(R^e)$ such that $\sigma m=f$ and $\sigma h=g$, which yields the explicit formula \eqref{e3} of $\sigma$. For (D2), it is easy to check that $\eta$ preserves the relations (R1). Thus $\eta$ can be extended to a unique $\sigma$-derivation on $R^e$. At last direct computation verifies (D3). 

Thirdly, we want to show that the right double Ore extension of $R^e$ associated to the DE-data is indeed a length two iterated Ore extension. 
Recall that, \cite[Definition 1.8]{ZZ:08}, $\sigma$ is said to be invertible if there exits another algebra map $\phi: R^e\to M_2(R^e)$ such that
\begin{align*}
\begin{pmatrix}
\phi_{11}& \phi_{12}\\
\phi_{21}& \phi_{22}
\end{pmatrix}
\begin{pmatrix}
\sigma_{11}& \sigma_{21}\\
\sigma_{12}& \sigma_{22}
\end{pmatrix}
=\begin{pmatrix}
\sigma_{11}& \sigma_{21}\\
\sigma_{12}& \sigma_{22}
\end{pmatrix}
\begin{pmatrix}
\phi_{11}& \phi_{12}\\
\phi_{21}& \phi_{22}
\end{pmatrix}
=
\begin{pmatrix}
Id&   0\\
0&Id
\end{pmatrix}.
\end{align*}

\begin{lemma}\label{lemma:invertible}
The algebra map $\sigma$ is invertible.
\end{lemma}
\begin{proof}
Similarly, we can define another algebra map $\phi: R^e\to M_2(R^e)$ such that
\begin{align*}
\phi(m_r)=
\begin{pmatrix}
m_r  & -m_{\alpha(r)} \\
0& m_r
\end{pmatrix},\quad
\phi(h_r)=
\begin{pmatrix}
h_r-m_{\alpha(r)}  & -h_{\alpha(r)}+m_{\alpha^2(r)}\\
0 & h_r-m_{\alpha(r)}
\end{pmatrix},
\end{align*}
for all $r\in R$. Direct computation shows that $\sigma$ and $\phi$ are inverse to each other.
\end{proof}
\begin{proposition}\label{Prop:U}
The right double Ore extension of $R^e$ is a length two iterated Ore extension such that 
\begin{align*}
R^e_P[y_1,y_2;\sigma,\eta]=R^e[y_1;\sigma_1',\eta_1'][y_2;\sigma_2',\eta_2'],
\end{align*}
where $\sigma_1'(a)=\sigma_{11}(a), \sigma_2'(a)=\sigma_{22}(a),\sigma_2'(y_1)=y_1$ and $\eta_1'(a)=\eta_1(a),\eta_2'(a)=\sigma_{21}(a)y_1+\eta_2(a),\eta_2'(y_1)=0$ for all $a\in R^e$.
\end{proposition}
\begin{proof}
By \eqref{e3}, we know $\sigma_{12}=0$ and $\sigma_{11}=\sigma_{22}$. Hence $\phi_{11}\sigma_{11}=\sigma_{11}\phi_{11}=Id$ by Lemma \ref{lemma:invertible}, which implies that $\sigma_{11}=\sigma_{22}$ is an automorphism of $R^e$. Then the right double Ore extension is an iterated Ore extension by \cite[Theorem 2.4]{CLM:11} (interchanging $y_1$ and $y_2$). The formulas of $\sigma_i'$ and $\eta_i'$ for $i=1,2$ are easy to deduce from the DE-data.
\end{proof}
Finally, we use the universal property to prove that the right double Ore extension $R^e_P[y_1,y_2;\sigma,\eta]$ is the universal enveloping algebra of $R[x;\alpha,\delta]_P$. We extend the two linear maps $m,h: R\to R^e$ to $R[x;\alpha,\delta]_P$ by
\begin{align}\label{e4}
m\left(\sum_{i=0}^{n} r_ix^i\right)&=\sum_{i=0}^{n}  m_{r_i}y_1^i,\\
h\left(\sum_{i=0}^{n}  r_ix^i\right)&=h_{r_0}+\sum_{i=1}^{n}  \left(im_{r_i}y_1^{i-1}y_2+(h_{r_i}+im_{\alpha(r_i)})y_1^{i}+im_{\delta(r_i)}y_1^{i-1}\right),\ \mbox{for all}\ r_i\in R,\notag
\end{align}
where we keep the same notations.

\begin{lemma}\label{PropertyP}
The triple $\left(R^e_P[y_1,y_2;\sigma,\eta],m,h\right)$ has property $\mathbf{P}$.
\end{lemma}
\begin{proof}
By (D1) and \eqref{e3}, we have the following identity in $R^e_P[y_1,y_2;\sigma,\eta]$: 
\begin{align}\label{e5}
y_2m_r&=m_ry_2+m_{\alpha(r)}y_1+m_{\delta(r)},\ \mbox{for all}\ r\in R.
\end{align}
Moreover, from $y_1h_r=h_ry_1+m_{\alpha(r)}y_1+m_{\delta(r)}$, inductively we have
\begin{align}\label{e6}
y_1^qh_r=(h_r+qm_{\alpha(r)})y_1^q+qm_{\delta(r)}y_1^{q-1},\ \mbox{for all}\ r\in R,q\ge1.
\end{align}

For condition (P1), by definition, we know that $\sigma=Id,\eta=0$ on $m_R$. So $y_1$ commutes with $m_R$. Then $m$ is a well-defined algebra map.

For condition (P3), it suffices to take $rx^p$ and $sx^q$ for any $r,s\in R$ and $p,q\ge 0$. We have
\begin{align*}
RHS\ of\ (P3)&=\left[pm_ry_1^{p-1}y_2+(h_r+pm_{\alpha(r)})y_1^p+pm_{\delta(r)}y_1^{p-1},\; m_sy_1^q\right] &\mbox{by \eqref{e4}} \\
&=(pm_ry_1^{p-1}y_2)(m_sy_1^q)-(m_sy_1^q)(pm_ry_1^{p-1}y_2)+(h_ry_1^p)(m_sy_1^q)-(m_sy_1^q)(h_ry_1^p) \\
&=pm_ry_1^{p-1}(m_sy_2+m_{\alpha(s)}y_1+m_{\delta(s)})y_1^q-pm_{rs}y_1^{p+q-1}y_2+h_rm_sy_1^{p+q}& \mbox{by \eqref{e5}}\\ &\quad-m_s\left((h_r+qm_{\alpha(r)})y_1^q+qm_{\delta(r)}{y_1}^{q-1}\right)y_1^{p}& \mbox{by \eqref{e6}}\\
&=\left(h_rm_s-m_sh_r+pm_{r\alpha(s)}-qm_{\alpha(r)s}\right)y_1^{p+q}+\left(pm_{r\delta(s)}-qm_{\delta(r)s}\right)y_1^{p+q-1}\\
&=m\left((\{r,s\}+pr\alpha(s)-q\alpha(r)s)x^{p+q}+(pm_{r\delta(s)}-qm_{\delta(r)s})x^{p+q-1}\right) &\mbox{by \eqref{e1}}\\
&=LHS\ of\ (P3).
\end{align*}

For condition (P4), we have
\begin{align*}
RHS\ of\ (P4)&=(m_ry_1^p)(qm_sy_1^{q-1}y_2+(h_s+qm_{\alpha(s)})y_1^q+qm_{\delta(s)}y_1^{q-1})+(m_sy_1^q)(pm_ry_1^{p-1}y_2\\&\quad+(h_r+pm_{\alpha(r)})y_1^p+pm_{\delta(r)}y_1^{p-1}) & \mbox{by \eqref{e4}}\\
&=(p+q)m_{rs}y_1^{p+q-1}y_2+(pm_{\alpha(r)s}+qm_{r\alpha(s)})y_1^{p+q}+(pm_{\delta(r)s}+qm_{r\delta(s)})y_1^{p+q-1}\\&\quad+m_ry_1^ph_sy_1^q+m_sy_1^qh_ry_1^p.
\end{align*}
Apply \eqref{e6} to the last two terms $m_r(y_1^ph_s)y_1^q$ and $m_s(y_1^qh_r)y_1^p$ of the above equality. Then
\begin{align*}
RHS\ of\ (P4)&=(p+q)m_{rs}y_1^{p+q-1}y_2+\left(m_rh_s+m_sh_r+(p+q)(m_{\alpha(r)s}+m_{r\alpha(s)})\right)y_1^{p+q}\\
&\quad+(p+q)(m_{\delta(r)s}+m_{r\delta(s)})y_1^{p+q-1}\\
&=(p+q)m_{rs}y_1^{p+q-1}y_2+\left(h_{rs}+(p+q)m_{\alpha(rs)}\right)y_1^{p+q}+(p+q)m_{\delta(rs)}y_1^{p+q-1}\\
&=h(rsx^{p+q})&\mbox{by \eqref{e4}}\\
&=LHS\ of\ (P4).
\end{align*}

For condition (P2), we need to show that $h$ is a Lie algebra map, i.e.,
\begin{align}\label{e7}
h(\{a,b\})=[h(a),h(b)]=h(a)h(b)-h(b)h(a),\ \mbox{for all}\ a,b\in R[x;\alpha,\delta]_p.
\end{align}
By definition, \eqref{e7} holds for $a,b\in R$. If $a=x$ and $b=r\in R$, then
\begin{align*}
LHS\ of\ \eqref{e7}=h(\alpha(r)x+\delta(r))=m_{\alpha(r)}y_2+(h_{\alpha(r)}+m_{\alpha^2(r)})y_1+m_{\delta\alpha(r)}+h_{\delta(r)}.
\end{align*}
On the other side, we have
\begin{align*}
RHS\ of\ \eqref{e7}=y_2h_r-h_ry_2=(h_{\alpha(r)}+m_{\alpha^2(r)})y_1+(h_r+m_{\alpha(r)})y_2+h_{\delta(r)}+m_{\delta\alpha(r)}-h_ry_2.
\end{align*}
Hence \eqref{e7} also holds in this case. It suffices to consider $a,b$ to be both monomials. For arbitrary $b$, we do inductions on the degree of $a$. It suffices to show that if \eqref{e7} holds for $h(\{c,b\}),h(\{d,b\})$, then it is true for $h(\{cd,b\})$. We get
\begin{align*}
&h(\{cd,b\})\\
=&h(\{c,b\}d+\{d,b\}c)\\
=&m(\{c,b\})h(d)+m(d)h(\{c,b\})+m(\{d,b\})h(c)+m(c)h(\{d,b\}) & \mbox{by (P4)}\\
=&m(\{c,b\})h(d)+m(d)\left(h(c)h(b)-h(b)h(c)\right)+m(\{d,b\})h(c)+m(c)\left(h(d)h(b)-h(b)h(d)\right)\\
=&(m(c)h(d)+m(d)h(c))h(b)-h(b)(m(c)h(d)+m(d)h(c))+m(\{b,c\})h(c)\\&+h(\{b,d\})h(d)+m(\{c,b\})h(d)+m(\{d,b\})h(c) &\mbox{by (P3)}\\
=&h(cd)h(b)-h(b)h(cd).
\end{align*}
This completes the proof.
\end{proof}
\begin{proof}[Proof of Theorem \ref{Thm:1}]
By Proposition \ref{Prop:U} and Lemma \ref{PropertyP}, it suffices to prove the universal property of  $R^e_P[y_1,y_2;\sigma,\eta]$ with two linear maps $m,h$ defined in \eqref{e4}. We work over the following commutative diagram with respect to property $\mathbf{P}$:
\[
\xymatrix{
R[x,\sigma,\delta]_P \ar[rr]^-{m}\ar[rr]_-{h}\ar[rd]^-{f}_-{g} &   &  R_P^e[y_1,y_2;\sigma,\eta]\ar@{-->}[dl]^-{\phi}\\
&    B  &\\
R\ar[rr]^-{m}\ar[rr]_-{h}\ar[uu]^-{i_R}\ar[ur]^-{f_R}_-{g_R}&     &R^e\ar[uu]_-{i_{R^e}}\ar@{-->}[ul]_-{\phi_R}
}.
\]
Let $(B,f,g)$ be another triple satisfying property $\mathbf{P}$. By precomposing $f$ and $g$ with the natural inclusion $i_R: R\to R[x;\alpha,\delta]_P$, we get two linear maps $f_R,g_R: R\to B$. It is clear that $(B,f_R,g_R)$ has property $\mathbf{P}$ induced from $(B,f,g)$. Hence, by the universal property of $R^e$, we obtain a unique algebra map $\phi_R$ from $R^e$ to $B$ such that $\phi_Rm=f_R$ and $\phi_Rh=g_R$.

Next we define an algebra map $\phi: R^e_P[y_1,y_2;\sigma,\eta]\to B$ by $\phi|_{R^e}=\phi_R$ and $\phi(y_1)=f(x)$ and $\phi(y_2)=g(x)$. In order to show that $\phi$ is well-defined, it suffices to check that $\phi$ preserves the relations in the DE-data. For (D2), we have
\begin{align*}
\phi(y_2y_1-y_1y_2)=g(x)f(x)-f(x)g(x)=f(\{x,x\})=0.
\end{align*}
For (D1), it is enough to take $a=m_r$ or $a=h_r$ for all $r\in R$. When $a=m_r$, we have
\begin{align*}
\phi\left[\begin{pmatrix}
y_1\\
y_2
\end{pmatrix}
m_r-\sigma(m_r)\begin{pmatrix}
y_1\\
y_2
\end{pmatrix}
-\eta(m_r)
\right]
=&\left[\begin{pmatrix}
\phi(y_1)\\
\phi(y_2)
\end{pmatrix}
\phi(m_r)-\phi(\sigma(m_r))\begin{pmatrix}
\phi(y_1)\\
\phi(y_2)
\end{pmatrix}
-\phi(\eta(m_r))
\right]\\
=&\left[\begin{pmatrix}
f(x)\\
g(x)
\end{pmatrix}f(r)
-\begin{pmatrix}
f(r)&0\\
f(\alpha(r))& f(r)
\end{pmatrix}
\begin{pmatrix}
f(x)\\
g(x)
\end{pmatrix}
-\begin{pmatrix}
0\\
f(\delta(r))
\end{pmatrix}
\right]\\
=&
\begin{pmatrix}
f(x)f(r)-f(r)f(x)\\
g(x)f(r)-f(\alpha(r))f(x)-f(r)g(x)-f(\delta(r))
\end{pmatrix}\\
=&\begin{pmatrix}
f(xr-rx)\\
f(\{x,r\}-\alpha(r)x-\delta(r))
\end{pmatrix}\\
=&0.
\end{align*}
When $a=h_r$, it is similar. Finally, the algebra map $\phi$ makes the above diagram commutate and the uniqueness of $\phi$ comes from the universal property of $R^e$.
\end{proof}

\begin{proof}[Proof of Corollary \ref{Cor:1}]
(1)-(4) is well-known for the properties of Ore extensions. Note that a connected graded algebra is twisted Calabi-Yau if and only if it is Artin-Schelter regular. So (5) and (6) comes from \cite{LWW:13}. And (7) follows from \cite[Corollary 1.3]{Phan}.
\end{proof}

\section{Examples}
As applications, we will consider examples of iterated quadratic Poisson-Ore extensions starting from a free quadratic Poisson algebra $\K[x_1,x_2,\cdots,x_n]$ with trivial Poisson bracket. Since the universal enveloping algebra of the initial Poisson algebra is a quadratic polynomial algebra with double-sized variables, all the properties in Corollary \ref{Cor:1} apply in all these cases.

\subsection{Semiclassical limits of quantized coordinate rings}
We explicitly treat one example: the coordinate rings of quantum matrices and their semiclassical limits. Let $\Char\K=0$, and $\K[[\hbar]]$ be the formal power series, where we define
\[
e(\alpha):=\sum_{i=0}^{\infty} \frac{\alpha^i}{i!}\hbar^i,
\]
for any $\alpha\in\K$. Given a nonzero scalar $\lambda\in \K^\times$ and a multiplicatively skew-symmetric matrix $p=(p_{ij})\in M_n(\K^\times)$. The multiparameter quantum $n\times n$ matrix algebra $B:=\mathcal O_{e(\lambda),e(p)}(M_n(\K[[\hbar]]))$ is the $\K[[\hbar]]$-algebra with generators $X_{ij}$ for $i, j=1,\cdots,n$, subject to 
\[
X_{lm}X_{ij}=
\begin{cases}
e(p_{li})e(p_{jm})X_{ij}X_{lm}+(e(\lambda)-1)e(p_{li})X_{im}X_{lj} & (l>i,m>j)\\
e(\lambda)e(p_{li})e(p_{jm})X_{ij}X_{lm} &(l>i,m\le j)\\
e(p_{jm})X_{ij}X_{lm}& (l=i,m>j).
\end{cases}
\]
Note that $B/\hbar B=\mathcal O(M_n(\K))$, the ordinary coordinate rings of matrix algebras. Denote the generators $x_{ij}:=X_{ij}+\hbar B$ in $\mathcal O(M_n(\K))$. The semiclassical limit process equips $\mathcal O(M_n(\K))$ with a Poisson bracket such that $\{x_{lm},x_{ij}\}=\overline{[X_{lm},X_{ij}]/\hbar}$. Explicitly,
\[
\left\{x_{lm},x_{ij}\right\}=
\begin{cases}
(p_{li}+p_{jm})x_{ij}x_{lm}+(\lambda-1)x_{im}x_{lj} & (l>i,m>j)\\
(\lambda+p_{li}+p_{jm})x_{ij}x_{lm} &(l>i,m\le j)\\
p_{jm}x_{ij}x_{lm}& (l=i,m>j).
\end{cases}
\]
After assigning a lexicographic order $x_{11}<x_{12}<\cdots<x_{nn}$ on the generators, we see that $\mathcal O(M_n(\K))$ is an iterated quadratic Poisson algebra of the form
\[
\mathcal O(M_n(\K))=\K[x_{11}][x_{12};\alpha_{12},\delta_{12}]_p\cdots[x_{nn};\alpha_{nn},\delta_{nn}]_p,
\]
where the derivations are given by
\[
\alpha_{lm}(x_{ij})
\begin{cases}
(p_{li}+p_{jm})x_{ij}& (l>i,m>j)\\
(\lambda+p_{li}+p_{jm})x_{ij} &(l>i,m\le j)\\
p_{jm}x_{ij} & (l=i,m>j).
\end{cases}
\]
And $\delta_{lm}(x_{ij})=(\lambda-1)x_{im}x_{lj}$ when $l>i,m>j$ and $\delta_{lm}(x_{ij})=0$ otherwise. More iterated quadratic Poisson algebras are provided in \cite[\S 2]{GL:11} through the semiclassical limit process. In a conclusion, we have the following:
\begin{proposition}
The Poisson universal enveloping algebras of the semiclassical limits of 
\begin{enumerate}
\item quantum affine spaces;
\item quantum matrices;
\item quantum symplectic and even-dimensional euclidean spaces;
\item quantum odd-dimensional euclidean spaces;
\item quantum symmetric and antisymmetric matrices
\end{enumerate}
are all Noetherian, Artin-Schelter regular and Koszul domains.
\end{proposition}
\begin{proof}
From \cite[\S 2]{GL:11}, we see that these quadratic Poisson algebras are all iterated Poisson-Ore extensions of the polynomial algebra $\K[x]$ with trivial Poisson bracket. Hence their universal enveloping algebras are iterated Ore-extensions of $\K[x,y]$ by Theorem \ref{Thm:1}. Then the result follows from Corollary \ref{Cor:1}.
\end{proof}

\subsection{Graded Poisson algebras with low ranked Poisson structures}
Throughout, let $A=\K[x_1,x_2,\cdots,x_n]$ be a graded Poisson algebra with $\deg x_i=1$ for all $1\le i\le n$. Assume the Poisson bracket of $A$ is given by
\begin{align*}
\{x_i,x_j\}=c_{ij}\left(x_1^2+x_2^2+\cdots+x_n^2\right),
\end{align*}
for all $1\le i,j\le n$ and $c_{ij}\in \K$. Note that the matrix of coefficients $C:=(c_{ij})$ is skew-symmetric, and the Poisson algebra will be denoted by $A(C)$. For the sake of simplicity, the base field $\K$ is algebraically closed and $\Char\K\neq 2$.
\begin{lemma}\label{lem3-1}
The following are equivalent:
\begin{enumerate}
\item $A(C)$ is a graded Poisson algebra.
\item The Jacobian identity holds for all elements in $A_1$.
\item $c_{ij}c_{ks}+c_{jk}c_{is}+c_{ki}c_{js}=0$ for all $1\le i<j<k<s\le n$.
\item $\rank (C)\le 2$.
\end{enumerate}
\end{lemma}
\begin{proof}
(1)$\Longleftrightarrow$(2) is well-known since $A$ is a free Poisson algebra generated in degree one. And (2)$\Longleftrightarrow$(3) follows from the Jacobi identity, since
\begin{align*}
\{x_i,\{x_j,x_k\}\}+\{x_j,\{x_k,x_i\}\}+\{x_k,\{x_i,x_j\}\}&=\{x_i,c_{jk}\omega\}+\{x_j,c_{ki}\omega\}+\{x_k,c_{ij}\omega\}\\
&=\sum_{1\le s\le n} 2c_{jk}\{x_i,x_s\}x_s+2c_{ki}\{x_j,x_s\}x_s+2c_{ij}\{x_k,x_s\}x_s\\
&=\sum_{1\le s\le n} 2\left(c_{jk}c_{is}+c_{ki}c_{js}+c_{ij}c_{ks}\right)x_s.
\end{align*}
Because $C$ is skew-symmetric, it suffices to consider that $1\le i<j<k<s\le n$.

It remains to show that (3)$\Longleftrightarrow$(4). If $\rank(C)\le 2$, then any order $4$ principal minors of $C$ are $0$.  Hence
for any $1\le i<j<k<s\le n$, we have 
\[0=
\begin{vmatrix}
 0 & c_{ij} & c_{ik} & c_{is} \\
 -c_{ij} & 0 & c_{jk} & c_{js} \\
 -c_{ik} & -c_{jk} & 0 & c_{ks} \\
 -c_{is} & -c_{js} & -c_{ks} & 0 
\end{vmatrix}=(c_{ij}c_{ks}+c_{jk}c_{is}-c_{ik}c_{js})^2.
\]
So $c_{ij}c_{ks}+c_{jk}c_{is}+c_{ki}c_{js}=0$. Conversely, if (3) holds, we should prove that $\rank(C)\le 2$. Suppose all $c_{ij}=0$, then we are done since $\rank(C)=0$. Otherwise, we may suppose $c_{12}\neq 0$.
Then 
\[
\begin{vmatrix}
 0 & c_{12} \\
 -c_{12} & 0 \\
\end{vmatrix}
=c_{12}^2\neq 0,
\]
which is a non-zero minor of order $2$ of $C$. Hence $\mathrm{rank}(C)\ge 2$. On the other hand, we have
 $c_{ij}=-c_{12}^{-1}(c_{j1}c_{i2}-c_{j2}c_{i1})$ by (3) for all $1\le i,j\le n$. Let $\begin{vmatrix}
c_{ir} & c_{is} & c_{it} \\
c_{jr} & c_{js} & c_{jt} \\
c_{kr} & c_{ks} & c_{kt}
\end{vmatrix}$
be any minor of $C$ of order $3$. Then we have 
\[
\begin{vmatrix}
c_{ir} & c_{is} & c_{it} \\
c_{jr} & c_{js} & c_{jt} \\
c_{kr} & c_{ks} & c_{kt} 
\end{vmatrix}
= \frac{-1}{c_{12}^3}
\begin{vmatrix}
c_{r1}c_{i2}-c_{r2}c_{i1} & c_{s1}c_{i2}-c_{s2}c_{i2} & c_{t1}c_{i2}-c_{t2}c_{i1} \\
c_{r1}c_{j2}-c_{r2}c_{j1} & c_{s1}c_{j2}-c_{s2}c_{j2} & c_{t1}c_{j2}-c_{t2}c_{j1} \\
c_{r1}c_{k2}-c_{r2}c_{k1} & c_{s1}c_{k2}-c_{s2}c_{k2} & c_{t1}c_{k2}-c_{t2}c_{k1} 
\end{vmatrix}
=0.
\]
This implies that $\mathrm{rank}(C)=2$.
\end{proof}

\begin{lemma}\label{lem3-2}
The following are equivalent:
\begin{enumerate}
\item $A(C)$ and $A(D)$ are isomorphic as graded Poisson algebras;
\item $C$ and $D$ are orthogonally similar;
\item $C$ and $D$ are similar.
\end{enumerate}
\end{lemma}
\begin{proof}
Let $\phi: A(C)\to A(D)$ be a graded Poisson isomorphism. Hence $\phi$ is given by some $n\times n$-matrix $M=(m_{ij})\in GL(n)$, since it is a graded algebra map. We denote $A(D)=\K[y_1,y_2,\cdots,y_n]$. Then $\phi(x_i)=\sum_{1\le j\le n}m_{ij}y_j$ for all $x_i$. Note that $\phi$ preserves the Poisson bracket, so we have
\begin{align}\label{f1}
\phi\left(x_1^2+x_2^2+\cdots+x_n^2\right)=\lambda\left(y_1^2+y_2^2+\cdots+y_n^2\right),
\end{align}
for some $\lambda\in \K^\times$. Direct computation shows that \eqref{f1} is equivalent to the condition $M^TM=\lambda Id$. Moreover, we have 
\begin{align}\label{f2}
\phi\left(\{x_i,x_j\}\right)=\left\{\phi(x_i),\phi(x_j)\right\}, 
\end{align} 
for all $x_i,x_j$. Then LHS of \eqref{f2} equals $\phi(c_{ij}(x_1^2+x_2^2+\cdots+x_n^2))=\lambda c_{ij}(y_1^2+y_2^2+\cdots+y_n^2)$. Meanwhile,
\begin{align*}
RHS\ of\ \eqref{f2}=&\{\sum_{1\le p\le n}m_{ip}y_p,\sum_{1\le q\le n}m_{jq}y_q\}\\
=&\sum_{1\le p,q\le n}m_{ip}m_{jq}\{y_p,y_q\}\\
=&\sum_{1\le p,q\le n}m_{ip}d_{pq}m_{jq}(y_1^2+y_2^2+\cdots+y_n^2).
\end{align*}
Hence \eqref{f2} is equivalent to the condition $MDM^T=\lambda C$. And (1) $\Longleftrightarrow$(2) follows by scaling $M$ by $1/\sqrt{\lambda}$.  Note that (2) $\Longleftrightarrow$(3) comes from \cite[Theorem 2.1]{DRZ:07}.
\end{proof}

\begin{proposition}
The graded Poisson algebra $A(C)$ is isomorphic to one of the following: 
\begin{enumerate}
\item The parametric family $A(a)$ with coefficient matrix:
$\begin{pmatrix}
0&\dots &\dots   &\dots  &0 \\
\vdots  &\ddots &&  & \\
\vdots& & \ddots&  & \\
\vdots& & &0  &a \\
0& & &  -a&0 
\end{pmatrix}$, where $a\in \K$ is the parameter;
\item the discrete class with coefficient matrix:
$\begin{pmatrix}
0&\dots&\dots&\dots&0\\
\vdots&0&-1   &i  &0 \\
\vdots  &1 &0&0  &-i \\
\vdots& -i& 0&  0&-1 \\
0& 0& i&1  &0
\end{pmatrix}$.
\end{enumerate}
Moreover, $A(a)\cong A(a')$ if and only if $a=\pm a'$. 
\end{proposition}
\begin{proof}
By Lemma \ref{lem3-1} and Lemma \ref{lem3-2}, it suffices to find normal forms for orthogonal similarity classes of skew-symmetric matrices, which have rank $\le 2$. Hence it follows from \cite[Theorem 2.5]{DRZ:07}. Additionally, $A(a)\cong A(a')$ if and only if the two matrices $\begin{pmatrix} 0 &a\\ -a&0\end{pmatrix}$ and $\begin{pmatrix} 0 &a'\\ -a'&0\end{pmatrix}$ are similar by Lemma \ref{lem3-2}. Then the result follows by computing the eigenvalues.
\end{proof}
Denote by $R=\K[y_1,y_2,\cdots,y_{n-1}]$ the free Poisson algebra with trivial Poisson bracket. Define the Poisson-Ore extension $R[y_n;\alpha,\delta]_P$, subject to 
\begin{align*}
\alpha=2aiy_{n-1}\frac{\partial}{\partial y_{n-1}},\quad \delta=2ai(y_1^2+\cdots+y_{n-2}^2)\frac{\partial}{\partial y_{n-1}}.
\end{align*}

\begin{proposition}
The parametric family $A(a)$ is isomorphic to $R[y_n;\alpha,\delta]_P$. Moreover, the universal enveloping algebra of $A(a)$ is a length two iterated Ore extension of the polynomial algebra on $2n-2$ generators, and it is a 
Noetherian, Artin-Schelter regular and Koszul domain.
\end{proposition}
\begin{proof}
In $A(a)$, we make a linear transformation such that $y_i=x_i$ for all $1\le i\le n-2$ and $y_{n-1}=x_{n-1}+ix_n$ and $y_n=x_{n-1}-ix_n$. Then it is easy to check that it is isomorphic to the described Poisson-Ore extension. Note that $R^e$ is the polynomial algebra on $2n-2$ generators by the construction of (R1). Then the remaining of the statement follows from Theorem \ref{Thm:1} and Corollary \ref{Cor:1}.
\end{proof}

\providecommand{\bysame}{\leavevmode\hbox to3em{\hrulefill}\thinspace}

\end{document}